\documentclass[a4paper,12pt]{amsart}
\usepackage{amsmath,amssymb,amsfonts,amsthm}
\usepackage[utf8]{inputenc}
\usepackage[T1]{fontenc}
\usepackage{tikz,relsize}
\usepackage{tikz-cd}
\usepackage[hmargin=3cm]{geometry}
\usepackage{url}
\usepackage{enumitem}
\usepackage{newfloat}
\DeclareFloatingEnvironment[fileext=lod,name=Diagram]{fdiagram}
\usepackage[mathscr]{eucal}

\newtheorem{thm}{Theorem}[section]
\newtheorem{lem}[thm]{Lemma}
\newtheorem{dfn}[thm]{Definition}
\newtheorem{prop}[thm]{Proposition}

\newtheorem{exl}[thm]{Example}
\newtheorem{rem}[thm]{Remark}

\newcommand{\Z}{{\mathbb Z}} 
\newcommand{\R}{{\mathbb R}}
\newcommand{\Q}{{\mathbb Q}}
\newcommand{\C}{{\mathbb C}}
\newcommand{\N}{{\mathbb N}}
\newcommand{\h}{{\mathbb H}}
\newcommand{\hT}{{\rm \h T}}
\newcommand{\holht}{\ensuremath{\mathcal HT}}
\newcommand{\CT}{{\rm \C T}}
\newcommand{\AIT}{{\rm \R T}}
\newcommand{\F}{{\mathbb F}}
\DeclareMathOperator{\PSL}{PSL}
\DeclareMathOperator{\SL}{SL}
\DeclareMathOperator{\Sp}{Sp}
\DeclareMathOperator{\SO}{SO}
\DeclareMathOperator{\GL}{GL}
\DeclareMathOperator{\End}{End}
\DeclareMathOperator{\Aut}{Aut}
\DeclareMathOperator{\Irr}{Irr}
\DeclareMathOperator{\res}{res}
\DeclareMathOperator{\Res}{Res}
\DeclareMathOperator{\ind}{ind}

\newcommand{\mcl}{\ensuremath{\text{McL}}}

\newcommand{\cconj}[1]{\overline{
\raisebox{0pt}[\dimexpr\height+1pt\relax]{\ensuremath{#1}}
}}

\newcommand{\ses}[7]{\ensuremath{#1 \longrightarrow #2 \stackrel{#3}{\longrightarrow} #4 \stackrel{#5}{\longrightarrow} #6 \longrightarrow #7}}

\begin{document}

\title[Flat quaternionic type manifolds]{Flat manifolds with holonomy representation of quaternionic type.}

\author{Gerhard Hiss, Rafa{\l} Lutowski and Andrzej Szczepa{\'n}ski}

\address{Gerhard Hiss: 
Lehrstuhl D f{\"u}r Mathematik, 
RWTH Aachen University,
52056 Aachen, Germany}
\email{gerhard.hiss@math.rwth-aachen.de}

\address{Rafa{\l} Lutowski, Andrzej Szczepa\'nski: 
Institute of Mathematics, 
University of Gda\'nsk, 
ul. Wita Stwosza 57, 
80-308 Gda\'nsk, Poland}
\email{rafal.lutowski@mat.ug.edu.pl, matas@ug.edu.pl}

\thanks{The third author was supported by FAPESP in Sao Paulo (Brazil) and MPI Bonn}


\date{\today}
\begin{abstract}
We are interested in the question of the existence of flat manifolds for which all $\R$-irreducible components of the
\emph{holonomy representation} are either absolutely irreducible, of complex  or of quaternionic type.
In the first two cases such examples are well known. But the existence of the third type of flat manifolds
was unknown to the authors. 
In this article we construct such an example. Moreover,
we present a list of finite groups for which a construction of manifolds of quaternionic type is
impossible.
\end{abstract}

\maketitle

\section{Introduction}
\label{Introduction}

Let $M^n$ be a flat Riemannian manifold of dimension $n$, i.e. a compact Riemannian manifold without boundary, with sectional curvature equal to zero.
From the Bieberbach theorem \cite{Sz12} it is known that the torsion free fundamental group $\Gamma = \pi_1(M^n)$ of the manifold $M^n$ defines a short
exact sequence
\begin{equation}\label{short}
0\to\Z^{n}\to\Gamma\stackrel{p}\to G\to 0,
\end{equation}
where $\Z^n$ is the unique maximal abelian subgroup of finite index in $\Gamma$. The finite group $G$ is the holonomy group of $M^n$.
By the \emph{holonomy representation} of the group $\Gamma$ we shall understand a homomorphism $\varphi_{\Gamma}\colon G\to \GL_n(\Z)$ given
by the formula
\[
\forall_{g\in G, z\in\Z^n} \varphi_{\Gamma}(g)(z) = \gamma z \gamma^{-1},
\]
where $\gamma \in \Gamma$ is any element such that $p(\gamma) = g$.
Since $\Z^n$ is a maximal abelian subgroup, $\varphi_{\Gamma}$ is a faithful representation.
Any short exact sequence (\ref{short}) corresponds to an element $\alpha_{\Gamma}\in H^2(G,\Z^n)$, where $G$ acts on $\Z^n$ via 
$\varphi_{\Gamma}.$
Moreover, an element $\alpha\in H^2(G,\Z^n)$ defines a flat manifold if and only if $\res^{G}_{C}\alpha\neq 0$ for any cyclic subgroup $C\subset G$.
Such an element $\alpha$ is called special, see \cite[p. 29, 36]{Sz12}.  
In this note we are interested in the question of the existence of flat manifolds for a special kind of holonomy representation.
We shall need some definitions.

Let $h\colon G\rightarrow \GL_m(\R)$ be a faithful representation.
Denote by $W$ a left $\R[G]$ module which is defined by the representation $h.$
We have a direct sum 
$W = V_1\oplus V_2\oplus \dots \oplus V_k$ 
of $\R[G]$-irreducible
modules $V_i$ for $i=1,2,...,k$.
It follows from \cite[Theorem (73.9)]{CR} that there are three kinds of summands: absolutely irreducible,
"complex" and "quaternionic" ones.
An irreducible $\R[G]$ module $V$ is "complex" if $\End_{\R[G]}(V) = \C$ (the complex numbers) and it is "quaternionic" if
$\End_{\R[G]}(V) = \h$ (the quaternionic field). If $\End_{\R[G]}(V) = \R$ then $V$ is absolutely irreducible.
Hence, we have a decomposition
\begin{equation}\label{decomposition} 
W = W_{\R}\oplus W_{\C}\oplus W_{\h} 
\end{equation}
into isotypic components of real, complex and quaternionic type.

There exists a similar characterization of $\C$-irreducible representations.
Let $\chi$ be character of $G$ and 
\[\nu_{2}(\chi) = \frac{1}{|G|}\sum_{g\in G}\chi(g^2).\]
By definition $\nu_2(\chi)$ is the Frobenius-Schur indicator. 

Let $V$ be as above an $\R[G]$-irreducible representation. Let $U$ be a simple submodule of the $\C[G]$-module $\C \otimes_\R V$. Then $\C \otimes_\R V = U$, which means that $V$ is absolutely irreducible or $\C \otimes_\R V = U \oplus \cconj{U}$, where $\cconj{U}$ denotes the conjugate module \cite[pages 106-109]{S77}. In the latter case, if $U\not\simeq\cconj{U}$ then $\End_{\R[G]}(V) = \C$ and if $U\simeq\cconj{U}$ then $\End_{\R[G]}(V) = \h$. Let $\chi_U$ denote the character of $U$. From \cite[p. 58]{Is76} we have
\[
\nu_2(\chi_{U}) = \left\{\begin{array}{rll}
1 & \mbox{, if $V$ is absolutely irreducible},\\
0 & \mbox{, if $V$ is of complex type},\\
-1 & \mbox{, if $V$ is of quaternionic type}.
\end{array}
\right.
\]

So knowing the irreducible representation on complex
spaces and their Frobenius-Schur indicators allows one to classify irreducible representations on real space.
  
In this paper we are interested in the question of the existence of flat manifolds whose $\R$-irreducible components of 
the holonomy representation
are either absolutely irreducible, or only of complex type or only of quaternionic type. Let us denote the
above classes of flat manifolds by $\AIT$, $\CT$ and $\hT$ respectively.

Equivalently, we are looking for a faithful $\Z G$ module $W$ such that $H^2(G,W)$ contains a special element and  $\R \otimes_\Z W = (\R \otimes_\Z W)_{F}$ for $F = \R, \C, \h$ respectively in the decomposition
(\ref{decomposition}).

In the first two cases (i.e. $F =\R$ and $F = \C$) such examples are well known and we shall present them in the next section. For example
any Hantzsche-Wendt manifold has holonomy representation whose $\R$-irreducible components are all absolutely irreducible.
Similarly, any flat manifold with holonomy group $(\Z_{3})^k$ and the first Betti number equal to zero and dimension
$2(2+k)$ (\cite[Th. 4.2]{Sz12}) has holonomy representation with $\R$-irreducible components only of complex type.

However in the case $F = \h$ such examples were unknown.
One of our main results (Theorem \ref{main}) is a construction of a flat manifold in the class $\hT$. In the final section we shall
give a list of finite groups which are not holonomy groups of flat manifolds in the class $\hT$.

An even dimensional flat manifold $M^{2n}$ is K\"ahler if and only if
all $\R$-irreducible summands of the holonomy representation $\varphi_{\pi_{1}(M^{2n})}$, 
which are also $\C$-irreducible occur with even multiplicity, see \cite[Prop. 7.2]{Sz12}. By the above definition
any flat manifold in the class $\CT$ is K\"ahler. On the other side,
any K\"ahler flat manifold with non-zero first Betti number  does not belong to the class $\CT.$

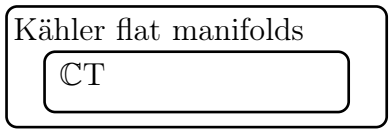
\begin{center}
\begin{tikzpicture}[rounded corners,line width=1]
\draw (-1,-.4) rectangle (3,.4);
\node at (-.5,.1) {$\C$T};

\draw (-1.5,-.6) rectangle (3.5,1);
\node at (0.5,.7) {K\"ahler flat manifolds};
\end{tikzpicture}
\end{center}

Similarly, any flat manifold $M^{4n}$ of dimension $4n$ of the class $\hT$
is hyperk\"ahler, i.e. $\varphi_{\pi_{1}(M^{4n})}(G)\subset \Sp(4n)$, see \cite{Hitchin}.
In fact, from (cf. \cite[p.58]{Is76}) it is well known that any irreducible complex representation $V$ with
$\nu_2(V) = -1$, has even dimension and carries a $G$-invariant, non-degenerate symplectic form.
However, any hyperk\"ahler flat manifold with the first Betti number different from zero is not $\hT$. 
See also \cite{W82} and \cite{DM01}.

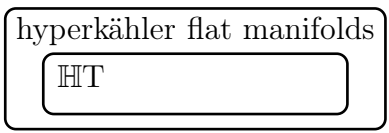
\begin{center}
\begin{tikzpicture}[rounded corners,line width=1]
\draw (-1,-.4) rectangle (3,.4);
\node at (-.5,.1) {$\hT$};

\draw (-1.5,-.6) rectangle (3.5,1);
\node at (1,.7) {hyperk\"ahler flat manifolds};
\end{tikzpicture}
\end{center}

For any finite group $G$, it has been observed (see \cite{John} and \cite[Prop. 3.2]{DM01}) that using "the double" 
construction it is possible to define  K\"ahler and hyperk\"ahler
manifolds with holonomy group $G$. 

\section{Real and complex case}
To begin with, we shall present a class of examples of flat manifolds 
whose \emph{holonomy representations} have a decomposition (\ref{decomposition}) with only absolutely irreducible components.
Any flat manifold with holonomy group $(\Z_2)^k, k\geq 2$
has such property. The simplest ones are the Klein Bottle
and the $3$-dimensional Hantzsche-Wendt manifold whose fundamental group is defined as follows, see \cite{Sz12}: 
\[
\Gamma = \langle\gamma_1 = (A, (1/2, 1/2, 0)), 
\gamma_2 = (B, (0, 1/2, 1/2))\rangle \subset \SO(3)\ltimes\R^3,
\]
where
$A =\left[
\begin{array}{lll}
1 & 0 & 0\\
0 &-1 & 0\\
0 & 0 &-1
\end{array}
\right] \text{ and } B =\left[
\begin{array}{lll}
-1 & 0 & 0\\
0 & 1 & 0\\
0 & 0 &-1
\end{array}
\right ].$

Let us present an example of a flat manifold $M_{k}$  where all components of the holonomy representation
are of "complex" type (have Frobenius-Schur indicator equal to $0$).
The manifold $M_k$ has holonomy group $(\Z_3)^k$, first Betti number zero and dimension equal to $2(k+2), k\geq 2$.

For example, to define the manifold $M_2$ we shall need matrices:
\[
D = \left[
\begin{array}{llll}
I_2 & 0 & 0 & 0\\
0 & C & 0 & 0\\
0 & 0 & C & 0\\
0 & 0 & 0 & C
\end{array}
\right], E = \left[
\begin{array}{llll}
C & 0 & 0 & 0\\
0 & I_2 & 0 & 0\\
0 & 0 & C & 0\\
0 & 0 & 0 & C^2
\end{array}
\right],
\text{ where }
C = \left[
\begin{array}{rr}
-1 & -1\\
1 & 0
\end{array}
\right ]
\]
and $I_k$ is the identity $k\times k$ matrix, for any $k \in \N$. Thus $D$ and $E$ are invertible integer matrices of degree $8$.
We define two elements
\[\gamma_{D} = (D,(-2/3,1/3,0,0,-2/3,1/3,-2/3,1/3)),\]
\[\gamma_{E} = (E, (0,0,-2/3,1/3,0,0,0,0))\]
of the group $SO(8)\ltimes\R^8$.
Let $\Gamma_{2} = \langle \gamma_{D},\gamma_{E},(I_8,x) : x\in \Z^8 \rangle$ and $M_2 = \R^8/\Gamma_2$. 
We have (see \cite{Sz12}) that $\pi_1(M_2) = \Gamma_2$.

\section{A flat manifold in the class $\hT$}

In this section we present one of the main results of our article.
\begin{thm}\label{main}
There exists a flat manifold with holonomy representation of quaternionic type. 
\end{thm}
\begin{proof}

Let the presentation of a group $G$ be given as follows. $G$ fits into the central extension
\[
\ses{1}{C_2^2}{}{G}{}{C_2^4}{1},
\]
it is generated by the elements $a,b,c,d$ of order $4$, such that $a^2=c^2$ and $b^2=d^2$ generate the center $Z(G)$ and we have commutator relations:
\[
\begin{array}{lll}
[a,b] = a^2 & [a,c] = a^2b^2& [a,d] = b^2\\
            & [b,c] = a^2   & [b,d] = a^2b^2\\
            &               & [c,d] = 1
\end{array}
\]
This is an example of a skew group, i.e. a group for which we have
\[
\forall_{\chi \in \Irr(G)} \chi(1) = 1 \vee \nu_2(\chi) = -1,
\]
where $\Irr(G)$ denotes the set of complex irreducible characters of $G$. In the library of small groups of GAP \cite{GAP16} $G$ is labeled as [64,245].

\begin{table}[ht]
\[
\begin{array}{l|rrrr}
       & 1 & a^2 & b^2 & a^2b^2\\ \hline
\chi_1 & 4 &   4 &  -4 &     -4\\
\chi_2 & 4 &  -4 &   4 &     -4\\
\chi_3 & 4 &  -4 &  -4 &      4\\
\end{array}
\]
\caption{List of nonzero values of characters $\chi_1,\chi_2,\chi_3$.}
\label{tab:char}
\end{table}

$G$ has exactly $3$ characters $\chi_1, \chi_2, \chi_3$ with Frobenius-Schur indicator equal to $-1$. They are distinguished by their values on the center only, since on every element of the set $G \setminus Z(G)$ each of them is equal to zero. The non-zero values are presented in Table \ref{tab:char}. Moreover, the characters are conjugate one to another, which means that for every $1 \leq i < j \leq 3$ there exists an automorphism $f_{ij} \in \Aut(G)$ such that $\chi_j = \chi_i \circ f_{ij}$. Let us introduce a notation for two of those automorphisms:
\begin{equation}
f_2 := f_{12}, f_3 := f_{13}.
\end{equation}
From Table \ref{tab:char} and the fact that the center $Z(G)$ is a characteristic subgroup of $G$ we have that
\[
f_2(b^2) = f_3(a^2b^2) = a^2.
\]

Now, for every $f \in \Aut(G)$, let $M^f$ denote the $G$-module twisted by $f$, i.e. $M^f = M$ and the action $\cdot_f$ comes from the original action of $G$ on $M$ as follows:
\[
g \cdot_f m = f(g) m,
\]
where $g \in G, m \in M^f$. Moreover, for every such automorphism $f$ and every subgroup $H$ of $G$ we have Diagram \ref{diag:commutativity}, where $(f_{|H})^*$ is an isomorphism. Hence, if we can find a $G$-lattice $M$ and an element $\alpha \in H^2(G,M)$ such that $\res_{C_1}\alpha \neq 0$, then
\[
\res_{C_i}f_i^*(\alpha)  = (f_{i|C_i})^*\res_{C_1}\alpha \neq 0,
\]
for $i=2,3$, where $C_i = \ker\chi_i$ is the group generated by $a^2,b^2$ and $a^2b^2$, for $i=1,2,3$ respectively. Hence the element
\begin{equation}
\label{eq:special}
\alpha + f_2^*(\alpha) + f_3^*(\alpha) \in H^2(G,M \oplus M^{f_2} \oplus M^{f_3})
\end{equation}
is special, since non-trivial elements of $Z(G)$ are the only elements of $G$ of order $2$.

\begin{fdiagram}
\[
\begin{tikzcd}
{H^{2}(G,M)} \arrow[r,"","f^*"] \arrow[d,"\res_{f(H)}"] 
& {H^2(G,M^f)} \arrow[d,"\res_{H}"] \\
{H^2(f(H),\res_{f(H)}M)} \arrow[r,"(f_{|H})^*"] 
& {H^2(H,\res_{H}M^f)}
\end{tikzcd}
\]
\caption{}
\label{diag:commutativity}
\end{fdiagram}

The construction of the lattice $M$ which works for the group $C_1$ in the above sense is as follows. Let $S_1 = \ind_{C_{1}}^{G} \Z$ 
be induced from the trivial $C_1$-module $\Z$. 
If we denote by $\chi$ the character of the $G$-module $S_1$, then direct calculations show that
\[
\langle \chi, \chi_1 \rangle = 4 \text{ and } \langle \chi, \chi_2 \rangle = \langle \chi, \chi_3 \rangle = 0.
\]
This means that there is a homogeneous component of the $\C G$-module $\C\otimes_\Z S_1$ of dimension $16$, which contains  quaternionic summands only.
Moreover, since $\chi_1$ is not the character of a real module, by \cite[Corollary 10.14]{Is76} we get that $2\chi_1$ is the character of some $\Q G$-module. If $\rho_1 \colon G \to \GL(32,\Z)$ is a permutation representation corresponding to the $G$-module $S_1$ (see below), then the procedure of finding homogeneous components in representations, presented in \cite[Section 2.6]{S77}, allows us to find a basis of the submodule of $S_1$ with character $4\chi_1$. To be more precise, 16 columns of the matrix
\[
B = \frac{2\chi_1(1)}{|G|} \sum_{g \in G} \cconj{2\chi_1(g)}\rho_1(g),
\]
up to multiplication by a rational constant, form a $\Z$-basis of a pure $\Z$-submodule of $S_1$, which admits a $G$ action. If we denote this module by $M$ then $M \cong \Z^{16}$ and computer calculations show that $H^2(G,M) = \langle \alpha \rangle \simeq \Z_2 $. Moreover, 
$\res^G_{C_1} \alpha \neq 0$
and the formula \eqref{eq:special} defines a special element which corresponds to the torsion-free extension
\[
\ses{0}{M \oplus M^{f_2} \oplus M^{f_3}}{}{\Gamma}{}{G}{1}.
\]
Since the character of the $G$-module $M \oplus M^{f_2} \oplus M^{f_3}$ is equal to $4(\chi_1+\chi_2+\chi_3)$, it must be faithful, cf. Table \ref{tab:char}, and $\Gamma$ is a Bieberbach group.

The representation $\rho_1$ may be defined as follows:
\[
\setlength{\arraycolsep}{.1em}
\begin{array}{lll}
a & \mapsto & P_{(1\,13)(2\,14)(3\,9)(4\,10)(5\,21)(6\,22)(7\,11)(8\,12)(15\,24)(16\,23)(17\,26)(18\,25)(19\,27)(20\,28)(29\,31)(30\,32)},\\
b & \mapsto & P_{(1\,3\,2\,4)(5\,18\,6\,17)(7\,15\,8\,16)(9\,14\,10\,13)(11\,24\,12\,23)(19\,30\,20\,29)(21\,25\,22\,26)(27\,32\,28\,31)},\\
c & \mapsto & P_{(1\,5)(2\,6)(3\,18)(4\,17)(7\,20)(8\,19)(9\,26)(10\,25)(11\,27)(12\,28)(13\,22)(14\,21)(15\,29)(16\,30)(23\,31)(24\,32)},\\
d & \mapsto & P_{(1\,8\,2\,7)(3\,15\,4\,16)(5\,19\,6\,20)(9\,23\,10\,24)(11\,14\,12\,13)(17\,30\,18\,29)(21\,28\,22\,27)(25\,32\,26\,31)},\\
\end{array}
\]
where $P_s$ denotes the permutation matrix of degree $32$ of an element $s$ of $S_{32}$, the symmetric group on $32$ letters.
 With this specific representation a basis of the module $M$ consists of the elements
\[
e_1-e_2,e_3-e_4,\ldots,e_{31}-e_{32},
\]
where $e_i$ denotes the $i$-th column of the identity matrix of degree $32$. In this case the structure of the $G$-module $M=\Z^{16}$ is given by the integral representation $\rho_M \colon G \to \GL(16,\Z)$ defined as follows
\[
\setlength{\arraycolsep}{.1em}
\def\arraystretch{1.3}
\begin{array}{lll}
a & \mapsto & P^{8,9,12,13}_{( 1\, 7)( 2\, 5)( 3\,11)( 4\, 6)( 8\,12)( 9\,13)(10\,14)(15\,16)},\\
b & \mapsto & P^{2,3,5,6,8,10,13,14}_{( 1\, 2)( 3\, 9)( 4\, 8)( 5\, 7)( 6\,12)(10\,15)(11\,13)(14\,16)},\\
c & \mapsto & P^{2,4,5,7,9,10,11,13}_{( 1\, 3)( 2\, 9)( 4\,10)( 5\,13)( 6\,14)( 7\,11)( 8\,15)(12\,16)},\\
d & \mapsto & P^{1,6,8,9,10,11,12,13}_{( 1\, 4)( 2\, 8)( 3\,10)( 5\,12)( 6\, 7)( 9\,15)(11\,14)(13\,16)},\\
\end{array}
\]
where the superscripts define rows of a permutation matrix which should be multiplied by $-1$.
The cocycle $\hat{\alpha}$ which defines non-zero cohomology class $\alpha$ in the group $H^1(G,\Q^{16}/\Z^{16}) \cong H^2(G,\Z^{16})$ is given by:
\[
\setlength{\arraycolsep}{.1em}
\renewcommand{\arraystretch}{1.3}
\begin{array}{lll}
a & \mapsto & \left( 0, \tfrac{1}{2}, 0, 0, 0, \tfrac{1}{2}, \tfrac{1}{2}, \tfrac{1}{2}, \tfrac{1}{2}, 0, \tfrac{1}{2}, 0, 0, \tfrac{1}{2}, \tfrac{1}{2}, 0 \right) + \Z^{16},\\
b & \mapsto & \left( 0, 0, 0, \tfrac{1}{2}, 0, \tfrac{1}{2}, 0, \tfrac{1}{2}, 0, \tfrac{1}{2}, 0, \tfrac{1}{2}, 0, \tfrac{1}{2}, \tfrac{1}{2}, \tfrac{1}{2} \right) + \Z^{16},\\
c & \mapsto & \left( 0, \tfrac{1}{2}, \tfrac{1}{2}, 0, 0, \tfrac{1}{2}, \tfrac{1}{2}, \tfrac{1}{2}, 0, \tfrac{1}{2}, 0, 0, \tfrac{1}{2}, 0, 0, \tfrac{1}{2} \right) + \Z^{16},\\
d & \mapsto & \left( 0, 0, 0, 0, 0, 0, 0, 0, 0, 0, 0, 0, 0, 0, 0, 0 \right) + \Z^{16}.\\
\end{array}
\]
One gets that $\rho_M(a^2) = I_{16}$, the identity matrix of degree $16$, and that
\[
\hat{\alpha}(a^2) = \left(\tfrac{1}{2},\tfrac{1}{2},\tfrac{1}{2},\tfrac{1}{2},\tfrac{1}{2},\tfrac{1}{2},\tfrac{1}{2},\tfrac{1}{2},\tfrac{1}{2},\tfrac{1}{2},\tfrac{1}{2},\tfrac{1}{2},\tfrac{1}{2},\tfrac{1}{2},\tfrac{1}{2},\tfrac{1}{2}\right) + \Z^{16},
\]
hence $\res_{C_1} \alpha \neq 0$.

The calculations are presented in more detail on the web page \cite{Lu19}.
\end{proof}

\section{Holonomy groups of flat manifolds from the class $\hT$}

Let $\holht$ denote the class of holonomy groups of manifolds from the class $\hT$. This section is motivated by the search for further examples in  \holht. We will present some necessary conditions for a group $G$ to belong to \holht. Then we will exclude certain simple groups from the class \holht. Finally, we comment on $2$-groups.

\subsection{Some necessary conditions}

\begin{dfn}[{\cite[Definition 4.2]{Sz12}}]
We say that a finite group is primitive if it is a holonomy group of a flat manifold with first Betti number equal to zero.
\end{dfn}

Since the holonomy representation of any flat manifold with non-zero first Betti is of non-quaternionic type,
a group in $\holht$ is primitive.
Moreover by \cite[Theorem 4.1]{Sz12} a finite group is primitive if and only if
no non-trivial cyclic Sylow $p$-subgroup of $G$ has a normal complement.

\begin{exl}
By \cite[Proposition 4.2]{Sz12} the following groups are non-primitive:
\begin{enumerate}[label=(\roman*)]
\item metacyclic groups $\Z_n\rtimes\Z_m, (m,n) = 1$;
\item the Borel subgroup $\Z_q\rtimes\Z_{q-1}\subset \SL_2(q)$, where $q$ is power of a prime $p$;
\item the group $\SL_2(3)$ (it has a normal $3$-complement).
\end{enumerate}
\end{exl}

\begin{prop}
\label{prop:htproperties}
Let $G \in \holht$. Then all of the following hold:
\begin{enumerate}[label=\arabic*.,ref=\arabic*]
\item \label{ht:even} 
	$G$ has even order.
\item \label{ht:non-abelian}
	$G$ is non-abelian.
\item \label{ht:center} 
	$Z(G)$ is an elementary abelian $2$-group.
\item \label{ht:primitive}
	If $p$ is a prime such that $p \mid G$ and a Sylow $p$-subgroup of $G$ is cyclic, then $G$ does not have a normal $p$-complement.
\item \label{ht:primitive2}
	The Sylow $2$-subgroup of $G$ is non-cyclic.
\item \label{ht:squareroots}
	Put $I(G):=|\{ g \in G \mid g^2=1 \}|$. Then \[I(G) < \sum_{\chi\in \Irr(G)}\chi(1) \text{ and } I(G) \leq |G|/2.\]
\item \label{ht:nonhomogeneous}
	$|\{ \chi \in \Irr(G) \mid \nu_2(\chi) = -1 \}| \geq 2$.
\item \label{ht:centerrep}
	Suppose $z \in Z(G)$ with $|z|=2$. Then there exist $\chi,\psi \in \Irr(G)$ with $\nu_2(\chi) = \nu_2(\psi) = -1$ and $\chi(z) = \chi(1), \psi(z)=-\psi(1)$. 
\end{enumerate}
\end{prop}

\begin{proof}
\begin{description}[before={\renewcommand\makelabel[1]{\bfseries Ad.\,##1:}}]
\item[\ref{ht:even}] By assumption the map  $l:G\to G, g\mapsto g^2$ is a bijection. Hence $\nu_2(\chi)\neq -1$ for any $\chi\in \Irr(G)$.
In fact $\nu_2(\chi) = \frac{1}{|G|}\Sigma_{g\in G}\chi(g^2) = \frac{1}{|G|}\Sigma_{g\in G}\chi(g) = \langle\chi, 1\rangle.$
Here $\langle,\rangle$ denotes the scalar product of characters, and $1$ the trivial character.

\item[\ref{ht:non-abelian}] The Frobenius-Schur indicator of every character of every abelian group is equal to $0$ or $1$.

\item[\ref{ht:center}] Suppose that $Z(G)$ contains an element $z$ with $|z|=4$ or $|z|=p$ for some odd prime $p$. Let $\psi$ be a faithful complex character of $G$. Then there is some irreducible constituent $\chi$ of $\psi$ such that $z^2$ is not in the kernel of $\chi$. As $\res_{Z(G)}\chi = \chi(1)\lambda$ for some $\lambda \in \Irr(Z(G))$, we obtain $\chi(z) \in \C \setminus \R$. Hence $\chi \neq \cconj{\chi}$ and so $G$ cannot have a faithful quaternionic representation.

\item[\ref{ht:primitive}] This is in fact a definition of a primitive group (see above).

\item[\ref{ht:primitive2}] This follows from \ref{ht:primitive} and \cite[Theorem 1]{T}.

\item[\ref{ht:squareroots}] By \cite[Lemma 2 on p. 252]{Wall} and \cite[p. 155]{Gruyter}, if any of given conditions does not hold then for every $\chi \in \Irr(G)$ we have $\nu_2(\chi)=1$.

\item[\ref{ht:nonhomogeneous}] The fact that the holonomy representation of a flat manifold with non-trivial holonomy group is non-homogeneous is the main result of \cite{Lu18}.

\item[\ref{ht:centerrep}] By assumption there exists a faithful $G$ lattice $L$ and a special element $\alpha\in H^2(G,L)$.
If $\chi(z) = \chi(1)$ for all $\chi \in \Irr(G)$, then $z$ is in the kernel of $\C\otimes_{\Z} L$, i.e. $z$ acts trivially on $L$,
and we have a contradiction. Suppose that $\chi(z) = -\chi(1)$ for all  $\chi \in \Irr(G)$. Then $z$ acts as $-1$ on  $\C\otimes_{\Z} L$,
hence $z$ acts as $-1$ on $L$. But then $H^2(G,L)$ cannot contain a special element, since $H^2(\langle z\rangle, \res_{\langle z\rangle}L) = 0$.
\end{description}
\end{proof}

\subsection{Some simple groups and their covering groups}

\begin{lem}\label{Remark2}
Let $G$ be a finite group and $p$ a prime number. 
Let $O_{p'}(G)$ denote the maximal normal subgroup of order prime to $p$.
Then $O_{p'}(G)$ is contained in the kernel of every $\chi\in\Irr(G)$
in the principal $p$-block. 
\end{lem}
\begin{proof}
Follows from \cite[Lemma (4.12)(ii)]{Feit}.
\end{proof}

\begin{lem}
\label{lem:principalpblock}
Let $G$ be a finite group and $L$ be a $G$-lattice. If $H^2(G,L)$ contains a special element, then for every prime divisor $p$ of $|G|$ there exists a constituent of $\C \otimes_\Z L$ which lies in the principal $p$-block of $G$.
\end{lem}
\begin{proof}
Follows from proofs of Lemmas 2.1 and 2.2.a of \cite{HS}.
\end{proof}

\begin{prop}
The following groups do not belong to \holht:
\begin{enumerate}[label=(\roman*),ref=\roman*]
\item \label{ht:s:sl}  $\SL_2(\F_q), \PSL_2(\F_q)$, where $q$ is a power of a prime;
\item \label{ht:s:sn}  $A_n, 2.A_n, S_n, 2.S_{n}, n \geq 5$;
\item \label{ht:s:pce} a perfect central extension of a sporadic simple group.
\end{enumerate}
\end{prop}

\begin{proof}
\begin{description}
\item[Ad \ref{ht:s:sl}] These groups are excluded by Lemma \ref{Remark2}, as all characters $\chi$ with $\nu_2(\chi) = -1$ satisfy $\chi(z) = -\chi(1)$
for $z\in Z(G), |z| = 2$ and hence $O_{p'}(G) \not \subset \ker \chi$ for any prime $p \neq 2$.

\item[Ad. \ref{ht:s:sn}] Let $n \geq 5$. Then $S_n$ and $A_n$ do not have irreducible representations of quaternionic type.
For example all $\C$-irreducible representations of $S_n$ are afforded by real representations, \cite[page 56, Corollary 4.15]{Is76}.
Similar properties have irreducible $\C$-representations of a dihedral group.
For the alternating group we shall use properties of complex irreducible representations of symmetric groups and Clifford theory, see \cite[page 79]{Is76}. Let $W$ be an $\R S_n$-irreducible module. We shall consider two cases.

\begin{description}
\item[Case 1] $\Res^{S_n}_{A_n}(W)\simeq V_1\oplus V_2$ where the $V_i$ are irreducible. 
That means $(\C\otimes V_1)\oplus (\C\otimes V_2)\simeq \Res^{S_n}_{A_{n}}(\C\otimes W).$
By Clifford's Theorem the $\C\otimes V_i, i = 1,2$ are
irreducible i.e. the $V_i$ are absolutely irreducible

\item[Case 2] $V:= \Res^{S_n}_{A_n}(W)$ is an irreducible $\R G$-module.Then:

\begin{enumerate}[label=(\roman*)]
\item $\C\otimes V$ is irreducible. This means that $V$ is absolutely irreducible.
\item $\C\otimes V = U \oplus \cconj{U}$. Since the index of $A_n$ in $S_n$ is equal $2$, then by \cite[Corollary 6.19]{Is76} $U \not\simeq \cconj{U}$.
By definition, we can conclude that $V$ is of "complex type". 
\end{enumerate}
\end{description}

For the groups $2.A_n$ and $2.S_n$ we argue as in the proof of \ref{ht:s:sl}.

\item[Ad. \ref{ht:s:pce}] Assume first that $G$ is simple. Using the Atlas \cite{ATL} we find that $\nu_2(\chi) \in \{ 0,1 \}$ for all $\chi \in \Irr(G)$ except if $G = \mcl$
and $\chi \in \{ \chi_{11}, \chi_{13} \}$. A calculation in GAP \cite{GAP16} shows that $\chi_{11}$ and $\chi_{13}$ do not belong to the principal $11$-block of $G$. Hence $\mcl \notin \holht$ by Lemma \ref{lem:principalpblock}. If $G$ is non-simple, by Proposition \ref{prop:htproperties}.\ref{ht:center} we get that $G = 2.S$ for a simple group $S \ncong \mcl$ (Schur multiplier of $\mcl$ is of order $3$). Then $G \notin \holht$ by Proposition \ref{prop:htproperties}.\ref{ht:centerrep}. This concludes the proof.
\end{description}
\end{proof}

\subsection{$2$-groups}

The following proposition shows the importance of the family of $2$-groups in our investigation.

\begin{prop}[{\cite[Theorem 1]{Mann} and \cite[Satz]{WW}}]
If a finite group $G$ is non-abelian and all its non-linear characters have Frobenius-Schur indicator equal to $-1$ then $G$ is a $2$-group.
\end{prop}

\begin{proof}
We follow \cite[p.354]{Mann}.
Let $G$ be a minimal counter example to the theorem.
First, since all non-linear irreducible characters have even degree, \cite[Theorem 1]{T} shows that $G$ has a normal $2$-complement $K$. Let $T$ be a Sylow $2$-subgroup of $G$. Then $T\simeq G/K$, therefore $G$ and $T$
have the same number of linear characters of order $2$ and all irreducible characters
of $T$ can be considered as characters of $G$. The Frobenius-Schur formula
\begin{equation}\label{invo}
I(G) = \Sigma_{\chi\in \Irr(G)}\nu_{2}(\chi)\chi(1),
\end{equation}
shows that $T$ has at least as many involutions as $G$ has. This is possible only
if $G$ and $T$ have the same number of involutions, and then (\ref{invo})
implies that all non-linear irreducible characters of $G$ are characters of $T,$
which means that $K$ is contained in the kernels of all these characters.
But then $K = 1$ and $G$ is a $2$-group.
\end{proof}

\begin{exl}
Let $E$ be an extraspecial $2$-group of order $2^{2m+1}, m \geq 1$. Then $E/E'$ is elementary abelian and there is a unique irreducible character of degree $2^m$ (see \cite[Problem 2.13]{Is76}).

There are two non-isomorphic such groups and $\nu_2(\chi)=-1$ in one of these copies (if $E$ is a central product of the quaternion group of order $8$ and several dihedral groups of order $8$).

In particular, $E$ does not belong to ${\mathcal HT}$ by the main result of \cite{Lu18}.
\end{exl}

\begin{rem}
Proposition \ref{prop:htproperties} may be used to show that the group $[64,245]$, defined in the proof of Theorem \ref{main}, is the only group of order less than or equal to $64$, which is the holonomy group of a flat manifold from the class $\hT$.
\end{rem}

\end{document}